\documentclass[11pt,reqno]{amsart}

\usepackage{amssymb}
\usepackage{amscd}
\usepackage{amsfonts}
\usepackage{setspace}
\usepackage{version}

\numberwithin{equation}{section}
   \theoremstyle{plain}
\newtheorem{theorem}[subsection]{Theorem}
\newtheorem{proposition}[subsection]{Proposition}
\newtheorem{lemma}[subsection]{Lemma}
\newtheorem{corollary}[subsection]{Corollary}
\newtheorem{definition}[subsection]{Definition}

\newtheorem*{main-corollary-repeat}{Corollary \ref{main-corollary}}

\renewcommand{\leq}{\leqslant}
\renewcommand{\geq}{\geqslant}
\newsavebox{\proofbox}
\savebox{\proofbox}{\begin{picture}(7,7)  \put(0,0){\framebox(7,7){}}\end{picture}}
\def\boxeq{\tag*{\usebox{\proofbox}}}

\newcommand\Z{\mathbb{Z}}
\newcommand\R{\mathbb{R}}

\newcommand\C{\mathbb{C}}
\newcommand\N{\mathbb{N}}

\newcommand\SL{\operatorname{SL}}
\newcommand\SU{\operatorname{SU}}
\newcommand\PU{\operatorname{PU}}

\newcommand\im{\operatorname{im}}
\newcommand\bes{\operatorname{bes}}

\newcommand\GL{\operatorname{GL}}

\newcommand\U{\operatorname{U}}

\newcommand\Mat{\operatorname{Mat}}

\newcommand\eps{\varepsilon}
\newcommand\id{I_n}
\def\proof{\noindent\textit{Proof. }}

\def\endproof{\hfill{\usebox{\proofbox}}\vspace{11pt}}

\begin{document}
\title[Approximate unitary groups]{Approximate groups III: the unitary case}

\author{Emmanuel Breuillard}
\address{Laboratoire de Math\'ematiques
Universit\'e Paris-Sud 11, 91405 Orsay cedex,
France }
\email{emmanuel.breuillard@math.u-psud.fr}
\author{Ben Green}
\address{Centre for Mathematical Sciences\\
Wilberforce Road\\
Cambridge CB3 0WA\\
England }
\email{b.j.green@dpmms.cam.ac.uk}

\begin{abstract}
By adapting the classical proof of Jordan's theorem on finite subgroups of linear groups, we show that every approximate subgroup of the unitary group $\U_n(\C)$ is almost abelian.
\end{abstract}
\maketitle
\tableofcontents

\section{Introduction}

This paper is the third in a series concerning \emph{approximate groups}, the first two papers in the series being \cite{bg1,bg2}. Let us begin by repeating the definition of
``$K$-approximate group'' due to T. Tao (see \cite{tao-noncommutative}).

\begin{definition}[Approximate groups]
Let $G$ be some group and let $K \geq 1$. A finite subset $A \subseteq G$ is called a $K$-approximate group\footnote{We make here a slight abuse of terminology, because our definition is not intrinsic and makes use of the ambient group $G$, so stricto sensu we define here approximate subgroups of $G$.} if

\begin{enumerate}
\item It is symmetric, i.e. if $a \in A$ then $a^{-1} \in A$, and the
identity lies in $A$;

\item There is a symmetric subset $X \subseteq G$ with $|X| \leq K$
such that $A^2 \subseteq XA$.
\end{enumerate}
\end{definition}

Here, as usual, $A^2$ denotes the product set $\{a_1a_2 | a_1,a_2 \in A\}$ and $XA$ denotes $\{ x a | x \in X, a \in A\}$. One of the main reasons for introducing approximate groups was to understand finite subsets $A$ in a group $G$ satisfying a doubling or tripling condition, that is $|A^2|$ or $|A^3|$ is not much larger that $|A|$. To a large extent, the classification of sets of small doubling reduces to the classification of approximate groups. For the relation between the two concepts, we refer the reader to Tao's original paper \cite{tao-noncommutative}. We have chosen to work with approximate groups here, rather than directly with sets of small tripling (say), so as to be compatible with our previous papers and other work of the authors and Tao. Approximate groups also have one or two advantages over sets with small tripling -- for example, they behave rather better under homomorphisms.

Working with approximate groups, it is convenient to introduce the following notion, defined by Tao in \cite{tao-solvable}.

\begin{definition}[Control]
Suppose that $A$ and $B$ are two subsets of a group $G$, and that $K
\geq 1$ is a parameter. We say that $A$ is $K$-controlled by $B$, or that $B$
$K$-controls $A$, if $|B| \leq K|A|$ and there is some set $X \subseteq G$ with $|X| \leq K$ and such that $A \subseteq XB \cap
BX$.
\end{definition}

Let $n \geq 1$ be an integer, and write $\U_n(\C)$ for the group of unitary matrices. The main result of this note is the following.

\begin{theorem}\label{mainthm}
Suppose that $A \subseteq \U_n(\C)$ is a $K$-approximate group and that $K \geq 2$. Then $A$ is $n^{Cn^3}K^{C n}$-controlled by $B$, a $K^C$-approximate subgroup of $\U_n(\C)$, which consists of simultaneously diagonalisable matrices.
\end{theorem}

Here $C$ is an absolute constant which could be specified explicitly if desired. As a corollary of Theorem \ref{mainthm} we can deduce a more precise result along similar lines, albeit with somewhat worse bounds.

\begin{corollary}\label{main-corollary}
Suppose that $A \subseteq \U_n(\C)$ is a $K$-approximate group, $K \geq 2$. Then there is a connected abelian subgroup $S \subseteq \U_n(\C)$ such that $A$ lies in the normaliser $N(S)$, and such that the image of $A$ under the quotient homomorphism $\pi : N(S) \rightarrow N(S)/S$ has cardinality at most $n^{Cn^4}K^{Cn^2}$.
\end{corollary}

From basic results on approximate groups, derived from the fundamental work of Ruzsa and developed in \cite{helfgott-sl2,tao-noncommutative}, we can also describe the subsets $A \subset U_n(\C)$ such that $|A^3| \leq K|A|$. They satisfy exactly the same conclusion as in the above corollary. The sets $A$ with $|A^2|\leq K|A|$, on the other hand, do not have such a nice structure, although it is a direct consequence of the above that such sets are contained in at most $n^{Cn^3}K^{Cn}$ cosets of some connected abelian subgroup $S \subseteq \U_n(\C)$ (see the remark after the proof of Corollary \ref{main-corollary}).

 In the case $n=2$, Bourgain and Gamburd \cite{bourgain-gamburd} proved a much stronger local version of Theorem \ref{mainthm} in which they considered covering numbers $\mathcal{N}(A,\delta)$ for every resolution $\delta>0$, instead of merely counting the number of points in $A$ as we do. However, their approached was based on the sum-product theorem (as used for example in the work of Helfgott \cite{helfgott-sl2}) and does not seem to extend easily to the higher rank case. See nevertheless the recent announcement \cite{bourgain-gamburd2}.

  From the qualitative point of view, a much more general result than Theorem \ref{mainthm} is contained in the work of Hrushovski \cite{hrushovski} and in later joint work of the authors\footnote{Added during revision: Pyber and Szab\'o have recently extended their approach in \cite{pyber-szabo}, which appeared simultaneously with \cite{bgt-announcement}, to cover fields of characteristic zero as well as finite characteristic as in their original work.} and Tao \cite{bgt-announcement,bgt-paper,freiman-glc-note}: in particular, the rough structure of approximate subgroups of $\GL_n(\C)$ is now understood, with the bounds in \cite{bgt-paper} being polynomial in $K$ for fixed $n$ just as in Theorem \ref{mainthm}. We have decided however that it is nonetheless worth having the present argument in the literature, since it is completely different to these other arguments and considerably more elementary in that nothing is required by way of algebraic group theory, quantitative algebraic geometry or model theory. Apropos of the last point, our bounds are completely explicit, whereas those in \cite{bgt-paper} are not on account of the use of ultrafilters there. As they stand, we believe that the methods of \cite{bgt-paper} would give a bound of the form $O_n(K^{C^{n^2}})$ in Theorem \ref{mainthm}, with the $O_n(1)$ being ineffective. In principle\footnote{In the most recent version of their preprint \cite{pyber-szabo}, Pyber and Szab\'o record effective versions of all the algebro-geometric arguments used in their work. It is not immediately clear to us exactly what explicit bound this could possibly give in our main theorem.} all uses of ultrafilters in our papers with Tao could be replaced by effective algebraic geometry arguments, thereby giving an explicit dependence on $n$; however it is extremely unlikely that in so doing one would beat the exponential dependence in $n$ that we have attained in Theorem \ref{mainthm}. Moreover, the power of $K$ appearing in Theorem \ref{mainthm} is merely linear in $n$ rather than exponential. We do not make any claim that this dependence is sharp -- indeed we believe it possible that a bound of the form $O_n(K^C)$ is the truth in Theorem \ref{mainthm}, where $C$ is independent of $n$. The
implied constant certainly cannot be independent of $n$ and most grow faster than exponentially even in the case $K = 1$. This can be seen by taking $A \cong \mbox{Sym}(n)$, the symmetric group on $n$ letters.

The proof in this paper can be viewed as an approximate version of the standard proof of Jordan's theorem on finite subgroups of $\U_n(\C)$, which states that such subgroups $G$ have an abelian subgroup $H$ with $[G : H] = O_n(1)$.


In addition to these remarks we note that Theorem \ref{mainthm} can be used as a substitute for the so-called \emph{Solovay-Kitaev argument} (see the appendix to \cite{solovay-kitaev}) which features in the variant of Kleiner's proof of Gromov's theorem on groups of polynomial growth due to Shalom and Tao \cite{shalom-tao}. In fact the arguments of our paper offer a new elementary proof of the fact, traditionally derived from the Tits alternative, that finitely generated subgroups of $\U_n(\C)$ with polynomial growth are virtually abelian. While the Tits alternative implies the exponential growth of non-virtually abelian subgroup of $\U_n(\C)$, our arguments fall short of this. They do however give a super-polynomial lower bound on the size of a word ball or radius $r$ of the type $\exp(r^\alpha)$ for some $\alpha=\alpha(n)>0$. We will remark further on this connection in \S \ref{gromov-sec}.

\emph{Notation.} The letters $c,C$ stand for absolute constants; different instances of the notation may refer to different constants. All constants in this paper could be specified explicitly if desired.

\emph{Acknowledgements.} The authors thank Terence Tao for helpful conversations. We are also massively indebted to the anonymous referee, who paid a startling amount of attention to our work and also suggested a simplification of our argument. Some of this work was completed while the authors were attending the MSRI workshop on Ergodic Theory and Harmonic Analysis in the autumn of 2008.

\section{On Jordan's Theorem}\label{jordan-sec}

In a sense, the main idea of our paper is to take a proof of Jordan's theorem on finite subgroups of $\U_n(\C)$ and then modify it so that it works in the context of approximate groups too (note that a subgroup is precisely the same thing as a $1$-approximate group).

\begin{theorem}[Jordan \cite{jordan}]
Suppose that $A$ is a finite subgroup of $\U_n(\C)$. Then there is an abelian subgroup $A' \subseteq A$ with $[A : A'] \leq F(n)$. We can take\footnote{That is to say, our proof gives a bound of this form. Completely optimal values of $F(n)$ are known by more sophisticated arguments, the latest of which, due to M. Collins \cite{collins}, give the sharp bound $F(n)\leq (n+1)!$ for $n$ large enough and make use of the classification of finite simple groups. The slightly strange use of the letter $A$ in this statement is so that it may easily be compared with results in later sections.} $F(n) = n^{Cn^3}$ for some absolute constant $C$.
\end{theorem}

Jordan's original proof was a very ingenious variation on the theme of the celebrated classification of Plato's solids, and was mainly algebraic. The proof we give here however is mainly geometric. It is a slight variant, which we learned from the weblog of T.~Tao \cite{tao-blog}, on the classical proof of Jordan's theorem given for instance in \cite{curtis-reiner}, itself based on arguments of Bieberbach and Frobenius (see \cite{bieberbach,frobenius}). The argument relies on the basic fact that the commutator of two elements close to the identity in a Lie group is itself much closer to the identity. This idea has been used repeatedly ever since (it is nowadays also sometimes referred to as the Zassenhaus-Kazhdan-Margulis trick, see \cite[chap. 8]{raghunathan}) and is also the main tool in the Solovay-Kitaev algorithm \cite{solovay-kitaev} mentioned above.

We remark that Jordan's theorem actually applies to finite subgroups of $\GL_n(\C)$, but the first step of the proof is to apply Weyl's unitary trick to reduce to the unitary case. No analogue of this trick appears to be possible in the context of approximate groups.

Suppose then that $A$ is a finite subgroup of $\U_n(\C)$. The key observation is the following very well-known fact. The reader may also find a nice explanation of cognate ideas in the proof of \cite[Lemma 4.7]{gowers-quasirandom}.

\begin{lemma}[Element with large centraliser]\label{comm-lem} Suppose that $A \subseteq \U_n(\C)$ is a finite group. Then at least one of the following holds:
\begin{enumerate}
\item There is a subgroup $A' \leq A$ consisting of scalar multiples of the identity with $[A : A'] \leq n^{Cn^2}$;
\item there is an element $\gamma \in A$, not a scalar multiple of the identity, whose centraliser $C_A(\gamma) := \{x \in A: x\gamma = \gamma x\}$ has cardinality at least $n^{-Cn^2} |A|$.
\end{enumerate}
\end{lemma}
\begin{proof} Equip $\Mat_n(\C)$ with the Hilbert-Schmidt norm: take some orthonormal basis $e_1,\dots,e_n$ for $\C^n$ and define $\Vert M \Vert := (\sum_{i,j} |m_{ij}|^2)^{1/2}$, where the $m_{ij}$ are the matrix entries of $M$ with respect to this basis. It is well-known that this is an algebra norm, that is to say $\Vert M_1 M_2 \Vert \leq \Vert M_1\Vert \Vert M_2 \Vert$, and we shall use this fact several times. Every unitary matrix has norm $\sqrt{n}$, and the Hilbert-Schmidt norm is invariant under left and right multiplication by unitary matrices. Let $d$ be the distance induced by this norm, that is to say $d(x,y) := \Vert x - y \Vert$.

We claim that $A'$, the subset of $A$ consisting of all elements with distance at most $1/4\sqrt{n}$ from the identity, has cardinality at least $n^{-Cn^2}|A|$. To see this, observe that a simple volume-packing argument implies that $U_n(\C)$ may be covered by $n^{Cn^2}$ balls of the form $\{g \in U_n(\C) : d(g,g_0) \leq 1/4\sqrt{n}\}$. At least one of these balls contains at least $n^{-Cn^2} |A|$ elements of $A$. However for all of these elements $g$ we have
\[ \Vert gg_0^{-1} - \id \Vert = \Vert g - g_0 \Vert  \leq \textstyle 1/4\sqrt{n}.\]
This is establishes the claim. We now distinguish two cases.

\emph{Case 1.} Every element of $A'$ is a scalar multiple of the identity. Then we clearly have alternative (i) in the statement of the lemma.

\emph{Case 2.} At least one element of $A'$ is not a multiple of the identity.  Let $\gamma \in A'$ be that amongst the elements of $A'$ which are not scalar multiples of the identity for which $d(\gamma,\id)$ is minimal. Then if $x \in A'$ is arbitrary we have
\begin{align*} d([\gamma,x] , \id) & = d(\gamma x \gamma^{-1} x^{-1},\id) \\ & = \Vert  (\gamma - \id)(x - \id) - (x-\id)(\gamma-\id)\Vert  \\ & \leq 2\Vert  \gamma - \id \Vert  \Vert  x-\id \Vert  \\ & \leq  d(\gamma,\id)/2.\end{align*}
Since $A$ is a group, the commutator $[\gamma,x]$ is an element of $A$. If it is a scalar multiple of the identity then, since $\det [\gamma, x] = 1$, we must have $[\gamma,x] = e^{2\pi i r/n}\id$ for some $r \in \N$.

Note that if $r \neq 0$ we have
\begin{equation}\label{comp} d(e^{2\pi i r/n} \id, \id) = | e^{2\pi i r/n} - 1| n^{1/2} \geq |\sin (\pi /n)| n^{1/2} \geq 2/\sqrt{n}.\end{equation}Since $d([\gamma, x], \id) \leq 1/4\sqrt{n}$ this implies that $[\gamma, x] = \id$. If $[\gamma, x]$ is not a scalar multiple of the identity then, by the asserted minimality of $d(\gamma,\id)$, we are also forced to conclude that $[\gamma,x] = \id$. In either case we have established that $x$ commutes with $\gamma$, and hence the whole of $A'$ lies in the centraliser $C_A(\gamma)$. This is option (ii) in the statement of the lemma.\end{proof}

Let us recall now the following standard fact.

\begin{lemma}[Centralizers]\label{centraliser-fact}
Let $\gamma \in \U_n(\C)$ is not a multiple of the identity. Then the centraliser $C_{\U_n(\C)}(\gamma)$ is isomorphic to a direct product $\U_{n_1}(\C) \times \dots \times \U_{n_k}(\C)$, where $n_1 + \dots + n_k = n$ and $n_i < n$ for all $i$.
\end{lemma}
\begin{proof}
The matrix $\gamma$, being unitary, is diagonalisable. Its centraliser in $\GL_n(\C)$ may therefore be identified with $\GL_{n_1}(\C) \times \dots \times \GL_{n_k}(\C)$, where $n_1 + \dots + n_k = n$ and $n_i < n$; the integers $n_i$ are of course the multiplicities of the eigenvalues of $\gamma$. It is clear that the intersection of such a block subgroup with $\U_n(\C)$ is precisely $\U_{n_1}(\C) \times \dots \times \U_{n_k}(\C)$, and this completes the proof.
\end{proof}

We may now complete a proof of Jordan's theorem, proceeding by induction on the rank $n$. Supposing that $A$ is a finite subgroup of $\U_n(\C)$, we apply Lemma \ref{comm-lem}. If option (i) holds then we are done; otherwise, option (ii) holds and we have a subgroup $Z$ (the centraliser in $A$ of some $\gamma$, not a scalar multiple of the identity) of size at least $n^{-Cn^2}|A|$ and which is isomorphic to a subgroup of $\U_{n_1}(\C) \times \dots \times \U_{n_k}(\C)$ where $k \leq n$ and $n_i < n$ for all $i$. Writing $\pi_i : Z \rightarrow \U_{n_i}(\C)$ for projection onto the $i$th factor, it follows from the induction hypothesis that there is an abelian subgroup $Z_i \subseteq \pi_i(Z)$ with $[\pi_i(Z) : Z_i] \leq F(n_i)$. The subgroup $B=\cap_i \pi_i^{-1}(Z_i)\subset Z$ satisfies
\[ |B| \geq \frac{|Z|}{F(n_1)\dots F(n_k)} \geq \frac{n^{-Cn^2}}{F(n_1)\dots F(n_k)}|A|,\]
and is abelian. 
So Jordan's theorem follows provided that
\[ F(n) \geq F(n_1)\dots F(n_k) n^{-Cn^2}.\] That a function of the form $F(n) = n^{Cn^3}$ satisfies this inequality is an immediate consequence of the following elementary lemma.\hfill $\Box$

\begin{lemma}\label{elementary-lem}
Suppose that $n \geq 2$ and that $n_1,\dots,n_k$ are positive integers with $n_i < n$ for all $i$ and $n_1 + \dots + n_k = n$. Then
\[ n^3 > n_1^3 + \dots + n_k^3 + n^2.\]
\end{lemma}
\proof
It is immediate by convexity or direct verification that $(x-1)^3 + (y+1)^3 > x^3 + y^3$ whenever $x,y$ are positive integers. Thus the maximum value of $n_1^3 + \dots + n_k^3$ subject to the constraint $n_1 + \dots + n_k = n$ occurs when $k = 2$ and $n_1 = n-1$, $n_2 = 1$. The result then follows immediately from the inequality
\begin{equation}\boxeq n^3 = (n-1)^3 + 1 + n^2(3 - \frac{3}{n}) > (n-1)^3 + 1 + n^2.\end{equation}

\section{Approximate subgroups of the unitary group}\label{sec3}

We turn now to the proof of Theorem \ref{mainthm}. We do this by modelling the proof of Jordan's theorem given in \S \ref{jordan-sec}, starting with Lemma \ref{comm-lem}, the lemma which located an element with large centraliser. We saw in Lemma \ref{comm-lem} that multiples of the identity were slightly troublesome. To ease these issues we work for now with the \emph{special} unitary group $\SU_n(\C) := \{g \in \U_n(\C) : \det g = 1\}$.

\begin{lemma}[Element with large centraliser]\label{comm-lem-2}
Suppose that $A \subseteq \SU_n(\C)$ is a $K$-approximate group with $|A| > n$. Then there is an element $\gamma \in A^2$ which is not a multiple of the identity and commutes with at least $n^{-Cn^2} K^{-6}|A|$ elements of $A^2$. \end{lemma}
\proof Since we are working in $\SU_n(\C)$, the only multiples of the identity are $e^{2\pi i r/n}\id$ with $r \in \N$. Since $|A| > n$, there certainly \emph{is} some $\gamma \in A^2$ which is not a multiple of the identity. Since $\gamma$ commutes with $\id$, which is an element of $A^2$, the lemma is trivial whenever $|A| \leq n^{Cn^2}$. Assume henceforth that
\begin{equation}\label{assumption} |A| > n^{Cn^2}.\end{equation} This is a variant of an argument pioneered by Solymosi \cite{solymosi} in the context of sum-product estimates for $\C$. For each $a \in A$ select an element $a^* \in A \setminus \{a\}$ which is nearest, or joint-nearest, to $a$ in the sense that $d(a , a^*) \leq d( a , a')$ for all $a' \in A$ (where $d$ is, as in the previous section, the distance induced by the Hilbert-Schmidt norm). Write $r_a := d(a,a^*)$. Consider the map $\psi : A \times A \times A \rightarrow A^2 \times A^2 \times A^2 \times A^2$ defined by
\[ \psi(a,a_1,a_2) := (a_1 a, a_1a^*, a a_2, a^* a_2).\]
It is certainly the case that $a_1 a$ is ``near''  $a_1 a^*$, and that $a a_2$ is ``near'' $a^* a_2$. If it was in fact the case that $a_1 a^*$ was the nearest point in $A^2$ to $a_1 a$, and $a^* a_2$ the nearest point in $A^2$ to $a a_2$, we would clearly have $|\im \psi| \leq |A^2|^2$. Since $|A^2| \leq K^2|A|$, it would follow that some fibre of $\psi$ has size at least $|A|/K$. But if $\psi(a,a_1,a_2) = \psi(b,b_1,b_2)$ then we have, of course, $a_1 a = b_1 b$, $a_1 a^* = b_1 b^*$, $a a_2 = b b_2$ and $a^* a_2 = b^* b_2$. Writing $\gamma := a^{-1} a^* = b^{-1} b^*$ we would have \[ a \gamma a^{-1} = a^* a^{-1} = (a^* a_2) (a a_2)^{-1} = (b^* b_2) (b b_2)^{-1} = b^* b^{-1} = b \gamma b^{-1}\] and hence $b^{-1}a \in C_G(\gamma)$. As a consequence,  $|C_G(\gamma) \cap A^2| \geq K^{-2}|A|$.

To turn this into a proof of the lemma we must resolve two issues. First, we need to ensure that $\gamma$ is not a multiple of the identity. Secondly and more seriously it will not, in general, be the case that $a_1 a^*$ is the nearest point in $A^2$ to $a_1 a$. Regarding this second point it turns out that something a little
weaker is true: for many triples $(a , a_1 , a_2 )$ there are not many points of $A^2$ closer to $a_1 a$ than $a_1 a^*$,  and not many points of $A^2$ closer to $a a_2$ than $a^* a_2$. In what follows, write $B_n$ for the \emph{weak Besicovitch constant} of $\Mat_n(\C)$; see Appendix \ref{bes-app} for a full discussion, and a proof that $B_n \leq C^{n^2}$.
We will examine \emph{well-behaved} triples $(a , a_1 , a_2)$ for which $a_1 a$ is ``almost'' the nearest neighbour of $a_1 a^*$ in $A^2$ in the sense that
\begin{equation}\label{3.1} U_{a,a_1} := |\{u \in  A^2 : d(a_1 a , u) \leq  r_a \}| \leq10 B_nK, \end{equation} for which $a a_2$ is ``almost'' the nearest neighbour of $a^* a_2$ in the sense that
\begin{equation}\label{3.2} V_{a,a_2} := |\{v \in A^2 : d(a a_2 ,v)\leq  r_a\}|\leq10 B_nK,\end{equation}
and for which
\begin{equation}\label{3.3} \mbox{$a^{-1} a^*$ is not a multiple of the identity}.\end{equation}
It is not obvious that there are any well-behaved triples, but we claim that this good behaviour
is quite generic in the sense that there are at least $|A|^3/2$ well-behaved triples.

Let us first count the triples $(a, a_1, a_2) \in A \times A \times A$ for which \eqref{3.3} is violated. Since we are working in $\SU_n(\C)$, the only multiples of the identity are $e^{2\pi i r/n}\id$ with $r \in \N$, and so by the same computation we used in \eqref{comp} we get $d(a^{-1} a^*, \id) \geq 2/\sqrt{n}$ and so $r_a = d(a, a^*) \geq 2/\sqrt{n}$. By a simple volume-packing argument the number of $a$ with this property is at most $n^{Cn^2}$, and so \eqref{3.3} is violated for at most $n^{Cn^2}|A|^2 < |A|^3/10$ triples $(a, a_1, a_2)$, this last inequality being a consequence of \eqref{assumption}.

Turning now to the examination of \eqref{3.1}, fix $a_1$. Then the open balls $B_{r_a}(a a_1 )$, $a \in A$, have the property that no centre $a a_1$ of one of these balls lies inside any other ball $B_{r_{a'}}(a'a_1)$. Indeed, if this were the case then we would have $d(a,a') < r_{a'}$, contrary to the assumption that $a^{\prime *}$ is the closest point of $A$ to $a'$. It follows from the definition of the weak Besicovitch constant $B_n$ that no point $u \in \Mat_n(\C)$ can lie in more than $B_n$ of these balls. It follows that
\[ \sum_{a} U_{a, a_1} \leq B_n|A^2| \leq B_nK |A|.\]
An essentially identical argument using \eqref{3.2} implies that
\[ \sum_{a} V_{a,a_2} \leq B_n|A^2| \leq B_nK |A|.\]
The number of pairs $(a, a_1)$ for which $U_{a,a_1} \geq 10 B_n K$ is thus at most $|A|^2/10$, as
is the number of pairs $(a, a_2)$ for which $V_{a ,a_2}\geq 10 B_nK$. It follows from this that there are at least $|A|^3/2$ well-behaved triples, as claimed.

Let us now consider the map $\psi$ defined above, \[ \psi(a,a_1,a_2) = (a_1 a, a_1 a^*, aa_2,a^* a_2),\] restricted to this set $S$ of at least $|A|^3/2$ well-behaved triples. We claim that $\im(\psi|_S)$ is reasonably small; this implies that $\psi$ has a large fibre, and we may then conclude as in the simplified sketch above.

Suppose, then, that $(x,y,z,w) \in \im(\psi|_S)$. There are at most  $|A^2|$ choices for $x$, and the same for $z$. Once these have been specified, consider the possible choices for $y$. Single out one of these, $\overline{y}$, corresponding to the well-behaved triple $(\overline{a},\overline{a}_1,\overline{a}_2)$ with $d( x , \overline{y}) = d( \overline{a} ,\overline{a}^*)$ maximal. Then for all permissible $y$ we have
\[ d(\overline{a}_1 \overline{a} , y) = d(x , y)\leq d( x ,\overline{y}) = d(\overline{a}_1 \overline{a}, \overline{a}_1 \overline{a}^*) = d( \overline{a} ,\overline{a}^*) = r_a.\]
Since $(\overline{a},\overline{a}_1,\overline{a}_2)$ is a well-behaved triple, it follows from \eqref{3.1} that there are at most $10 B_n K$ choices for $y$. Similarly, there are at most $10 B_n K$ choices for $w$. It follows that
\[ |\im(\psi|_S)| \leq (10 B_n K)^2 |A^2|^2,\] and so $\psi$ has a fibre of size at least $C^{-n^2}K^{-6}|A|$. By precisely the same argument used in the informal discussion at the start of the proof, this implies the result.\endproof\vspace{8pt}

\emph{Added in revision.} Upon seeing our paper, and in particular noting our idea of mimicing the proof of Jordan's theorem in the approximate group setting, the referee came up with an elegant and simpler argument for proving (a very slight variant of) this pivotal lemma which he or she was generous enough to share with us. We sketch this now. First of all look at $A'$, the elements of $A^2$ at distance at most $1/4\sqrt{n}$ from the identity. As remarked above, none of these are multiples of the identity. By a simple volume-packing argument, $|A'| \geq n^{-Cn^2}|A|$. Let $\rho$ be the minimum value of $d(\gamma, \id)$ over all $\gamma \in A^{\prime}$. Suppose that there are $L$ elements $\gamma' \in A^{\prime 4}$ with \begin{equation}\label{gamma-sat} d(\gamma',\id) < \textstyle\frac{1}{2}\rho.\end{equation} Then, multiplying by the elements of $A'$ and using the minimality of $\rho$, we obtain the inequality $|A^{\prime 6}| \geq L|A'|$. Since $A$ is a $K$-approximate group we have $|A^6| \leq K^5 |A|$, and therefore $L \leq n^{Cn^2} K^5$.

However, by the inequalities noted in Case 2 of the proof of Lemma \ref{comm-lem}, any commutator $\gamma' = [\gamma,x]$, $x \in A'$, will satisfy \eqref{gamma-sat}. It follows that there are merely $n^{Cn^2} K^5$ different values taken by this commutator, and hence there is some further set $A'' \subseteq A'$, $|A''| \geq n^{-Cn^2} K^{-5}|A|$, such that $[\gamma, x] = [\gamma, y]$ whenever $x,y \in A''$. A very short computation confirms that $x^{-1}y$ centralises $\gamma$ for any such pair $x,y$, and this concludes the proof.\vspace{8pt}

\emph{Remark.} We note that this argument of the referee uses slightly less than our original one, in that only the bounded doubling of balls in the Hilbert-Schmidt norm is required, as opposed to the rather more subtle Besicovitch property. \vspace{8pt}

A consequence of Lemma \ref{comm-lem-2}, proven below, is the following.

\begin{corollary}\label{corol}
Suppose that $A \subseteq \U_n(\C)$ is a $K$-approximate group. Then either there is a coset $xZ$ of the centre $Z\cong \U_1(\C)\subset \U_n(\C)$ such that $|A\cap xZ|\geq n^{-1}|A|$, or there is an element $\gamma \in A^2$ which is not a multiple of the identity and commutes with at least $n^{-Cn^2} K^{-11}|A|$ elements of $A^2$.
\end{corollary}

In the proof of this corollary and elsewhere we require two lemmas concerning the behaviour of approximate groups under intersections and homomorphisms. Related results appear in work of Helfgott \cite{helfgott-sl3} and later papers such as \cite{bg1,bg2,tao-solvable}.

\begin{lemma}\label{basic-group-fact}
Let $K\geq 2$ be a parameter and let $A$ be a $K$-approximate subgroup of $G$. Let $H \leq G$ be a subgroup. Then $A^2 \cap H$ is a $2K^3$-approximate group and $|A^{k} \cap H| \leq K^{k-1}|A^2 \cap H|$ for every $k\geq 1$.
\end{lemma}
\begin{proof} Let $X$, $|X| \leq K$, be as in the definition of approximate group. Then, for any positive integer $k$, we have \begin{equation}\label{containment}A^{k} \subseteq X^{k-1}A.\end{equation}
Now if $g \in G$  and $y_1,y_2 \in gA \cap H$ then $y_1^{-1} y_2 \in A^2 \cap H$. It follows that
\[ gA \cap H \subseteq y(A^2 \cap H)\]
for any $y \in gA \cap H$ (or, if $gA \cap H$ happens to be empty, for any $y$ at all).
Let $Y$ be a set consisting of one such value of $y$ for each choice of $g \in X^{k-1}$. It follows from the preceding discussion and \eqref{containment} that
\[ A^k \cap H \subseteq Y(A^2 \cap H).\] This confirms the second statement of the lemma. Taking $k = 4$ and noting that $(A^2 \cap H)^2 \subseteq A^4 \cap H$ gives
\[ (A^2 \cap H)^2 \subseteq Y(A^2 \cap H).\]
Since $A^2 \cap H$ is symmetric, this implies that
\[ (A^2 \cap H)^2 \subseteq (A^2 \cap H)Y^{-1}.\]
This confirms that $A^2 \cap H$ is a $2K^3$-approximate group, with covering set $Y \cup Y^{-1}$.
\end{proof}
\begin{lemma}\label{lift-lemma}
Suppose that $A$ is a symmetric set in some group $G$, and that $\pi : G \rightarrow G'$ is a homomorphism from $G$ into some other group $G'$. Suppose that $X \leq G'$ is a set and that $|\pi(A) \cap X| = \delta |\pi(A)|$. Then $|A^3 \cap \pi^{-1}(X)| \geq \delta |A|$.
\end{lemma}
\begin{proof}
Let $M$ be the size of the largest fibre of $A$ above $G'$, that is to say $\max_x |A \cap \pi^{-1}(x)|$. Then $A^2$ has a fibre of size at least $M$ over $\mbox{id}_{G'}$, and thus $A^3$ has a fibre of size at least $M$ over each point of $\pi(A)$. In particular,
\[ |A^3 \cap \pi^{-1}(X)| \geq M|\pi(A) \cap X| \geq M\delta |\pi(A)|.\]
On the other hand it is clear that $|A| \leq M|\pi(A)|$. Combining these two inequalities leads to the stated bound.
\end{proof}

\noindent \emph{Proof of Corollary \ref{corol}.} Let $\pi$ be the projection $\U_n(\C) \rightarrow \PU_n(\C)$ whose kernal $\ker(\pi)=Z$ is the centre of $\U_n(\C)$. Let $A':=\pi^{-1}(\pi(A))\cap \SU_n(\C)=AZ \cap \SU_n(\C)$. Note that $|A'|\geq |\pi(A)|$. If $|\pi(A)|\leq n$, then there is a coset $xZ$ such that $|A \cap xZ| \geq n^{-1}|A|$. If not then $|A'| > n$ and so Lemma \ref{comm-lem-2} applies (with $A'$ in place of $A$) and we obtain an element $\gamma$ in $A^{\prime 2}$, not a multiple of the identity, such that
\[ |C_{\U_n(\C)}(\gamma) \cap A^{\prime 2}| \geq n^{-Cn^2} K^{-6}|A'|.\]
Pushing this forward under $\pi$ and noting that fibres of $\pi$ in $\SU_n(\C)$ have size at most $n$, we obtain
\[ |\pi(A^2)\cap \pi(C_{\U_n(\C)}(\gamma))|\geq n^{-C'n^2} K^{-6}|\pi(A)|.\] It follows from Lemma \ref{lift-lemma} that \[ |A^6 \cap C_{\U_n(\C)}(\gamma)|\geq n^{-C'n^2} K^{-6}|A|,\] and hence from Lemma \ref{basic-group-fact} that
\[ |A^2 \cap C_{\U_n(\C)}(\gamma)|\geq n^{-C'n^2} K^{-11}|A|.\]
This concludes the proof.\endproof

We have established an ``approximate'' analogue of Lemma \ref{comm-lem}. It remains to complete the proof of Theorem \ref{mainthm}, and we do this by proceeding in a manner rather analogous to that at the end of \S \ref{jordan-sec}, that is to say by induction on $n$.

To make this work efficiently, we prove the following statement.

\begin{lemma}\label{inductive-step}
Suppose that $A$ is a $K$-approximate subgroup of some group $G$ group \textup{(}which, in applications, will be a unitary group\textup{)}. Let $H \leq G$ be a subgroup isomorphic to $\U_{n}(\C) \times H_0$ for some group $H_0$ and some $n \geq 2$, and suppose that $|A \cap H x| \geq \delta |A|$ for some $\delta > 0$ and some coset $H x$. Then there is a further subgroup $H'$, isomorphic to $\U_{n_1}(\C) \times \dots \times \U_{n_k}(\C) \times \U_1(\C) \times H_0$ where $n_i < n$ for all $i$ and $n_1 + \dots + n_k = n$, together with an $x'$ such that $|A \cap H' x'| \geq n^{-Cn^2} \delta K^{-C}|A|$.
\end{lemma}

\emph{Proof of Lemma \ref{inductive-step}.} The hypothesis $|A \cap Hx| \geq \delta |A|$ immediately implies that $|A^2 \cap H| \geq \delta|H|$. By Lemma \ref{basic-group-fact} we see that $S := A^2 \cap H$ is a $K^3$-approximate group. By assumption we have
\[ H \cong \U_n(\C) \times H_0.\] The projection $\pi(S)$ onto the first factor $\U_n(\C)$ is another $K^3$-approximate group and we may apply Corollary \ref{corol} to it. If we are in the first case of that corollary, the lemma follows immediately with $H' = Z \times H_0$, where $Z\cong \U_1(\C)$ is the centre of $\U_n(\C)$.

If we are in the second case, then there is an element $\gamma$ in $S^2$ such that $|\pi(S)^2 \cap C_{\U_n(\C)}(\gamma)| \geq n^{-Cn^2} K^{-C}|\pi(S)|$.
 By Lemma \ref{centraliser-fact} this centraliser $C_{\U_n(\C)}(\gamma)$ is isomorphic to a subgroup of some product $\U_{n_1}(\C) \times \dots \times \U_{n_k}(\C)$ with $n_1 + \dots + n_k = n$ and $n_i < n$ for all $i$.  Write \[ H' := \U_{n_1}(\C) \times \dots \times \U_{n_k}(\C) \times \U_1(\C) \times H_0.\] By Lemma \ref{lift-lemma} we have $|S^6 \cap H'| \gg n^{-Cn^2} K^{-C}|S|$ and hence, by Lemma \ref{basic-group-fact}, that \[K^{11}|A^2 \cap H'|\geq |A^{12} \cap H'| \geq n^{-Cn^2} K^{-C}|S| \geq \delta  K^{-C}|A|.\] Since $A$ is a $K$-approximate group, $A^{2}$ is covered by $K$ translates $Ax$ of $A$. The result follows immediately.\endproof

\emph{Proof of Theorem \ref{mainthm}.} Simply apply Lemma \ref{inductive-step} repeatedly, starting with $H = \U_n(\C)$. After at most $n$ steps  we end up with some $x$ such that $|A \cap H'x| \geq n^{-Cn^3} K^{-Cn}|A|$, where $H'$ is isomorphic to a product of at most $2n$ copies of $\U_1(\C)$ and in particular is abelian. It follows that $|A^2 \cap H'| \geq n^{-Cn^3}K^{-Cn}|A|$, and hence by Lemma \ref{basic-group-fact} that $B := A^2 \cap H'$ satisfies the conclusions of Theorem \ref{mainthm}.
\endproof

\section{A more precise result}

Our aim in this section is to establish Corollary \ref{main-corollary}, a somewhat more precise structural conclusion about approximate subgroups of the unitary group. Let us begin by recalling the statement.

\begin{main-corollary-repeat}
Suppose that $A \subseteq \U_n(\C)$ is a $K$-approximate group. Then there is a torus $S \subseteq \U_n(\C)$ such that $A$ lies in the normaliser $N(S)$, and such that the image of $A$ under the quotient homomorphism $\pi : N(S) \rightarrow N(S)/S$ has cardinality at most $n^{Cn^4}K^{Cn^2}$.
\end{main-corollary-repeat}

Recall that by a \emph{torus} we mean a connected abelian subgroup. We will find it convenient to introduce the notion of \emph{root torus}: a root torus is by definition the intersection of conjugates of the full diagonal subgroup $T$ of $\U_n(\C)$. A root torus is \emph{a priori} a closed abelian subgroup of $\U_n(\C)$. It is in fact connected, as the following lemma shows.

\begin{lemma}\label{gtorus}
Every root torus in $\U_n(\C)$ is connected and hence is a torus. Moreover it is the intersection of at most $n$ conjugates of the full diagonal subgroup $T$.
\end{lemma}
\newcommand\diag{\operatorname{diag}}
\begin{proof}
Let $T_i=g_iTg_i^{-1}$ be a collection of conjugates of the full diagonal subgroup $T$ (say with $T_1=T$). Pick an element $\gamma \in T$ with distinct eigenvalues. An element of $\U_n(\C)$ lies in $T$ (resp. $T_i$) iff it commutes with $\gamma$ (resp. $g_i\gamma g_i^{-1}$). On the other hand a diagonal matrix $\diag(\lambda_1,...,\lambda_n)$ commutes with a matrix $(a_{ij})$ if and only  $\lambda_i=\lambda_j$ whenever $a_{ij}\neq 0$. From these remarks it follows that the intersection $\bigcap_iT_i$ is the subset of $T$ defined by the equality of certain eigenvalues. It is thus isomorphic to a direct product of at most $n$ copies of the group of complex numbers of modulus one, and in particular it is connected. The second assertion of the lemma also follows immediately.
\end{proof}

\noindent \emph{Proof of Corollary \ref{main-corollary}.} By our main theorem, there is a conjugate $T$ of the full diagonal subgroup of $\U_n(\C)$ and a $K^{C}$-approximate group $B \subseteq T$ which $n^{Cn^3}K^{Cn}$-controls $A$. In particular, $|A^2 \cap T| \geq \delta|A|$, where $\delta :=n^{-Cn^3}K^{-Cn}$. Let $S:=\bigcap_{a \in \langle A \rangle}aTa^{-1}$. Clearly $S$ is a root torus and $A$ lies in $N(S)$. Moreover, since $g$-tori are connected and the dimension of $S$ is at most $n$, there must exist $a_i \in A^n$, $i=1,...,n$, $a_1=\id$, such that $S=\bigcap_{i=1}^n a_i T a_i^{-1}$.

Set $S_i=\bigcap_{j < i} a_i T a_i^{-1}$. We will establish by induction that $|A^2 \cap S_i| \geq \delta_i |A|$, where $\delta_i=(\delta K^{-2n-6})^{i-1}$. This statement in the case $i=n+1$ implies that $|A^2 \cap S| \geq \delta_{n+1} |A|$. This establishes the corollary since $|\pi(A)\Vert A^2 \cap S| \leq |A^3|\leq K^2 |A|$, and so $|\pi(A)| \leq \delta_{n+1}^{-1}K^2 \leq n^{Cn^4} K^{Cn^2}$.

The base of the induction $i=2$ has already been checked, so we will focus on the induction step, assuming $i\geq 2$.
Set $B_1=A^2 \cap S_i$ and $B_2=a_i (A^2 \cap T) a_i^{-1} \subseteq A^{2n+2}$. We have $|B_2|=|A^2 \cap T| \geq \delta |A|$. On the other hand $B_1B_2 \subseteq A^{2n+4}$ and $|B_1| \geq \delta_i |A|$ by the induction hypothesis. It follows that if $F$ is the largest fibre of the map $\phi : B_1 \times B_2 \rightarrow B_1B_2$ defined by $\phi(b_1,b_2) =b_1b_2$, then
\[ |B_1\Vert B_2| \leq |F\Vert B_1B_2| \leq |F\Vert |A^{2n + 4}| \leq K^{2n + 3}|A|\] and therefore \[ |F|\geq \delta \delta_i K^{-2n - 3}|A|.\]  Since $F$ is a fibre of $\phi$,  there is $x \in B_1B_2$ with at least $|F|$ representations as $b_1b_2$ with $b_1 \in B_1$ and $b_2 \in B_2$. Fix one of these representations and let $x = b'_1 b'_2$ be any other. Then we clearly have
\[ b_1^{\prime -1} b_1 = b'_2 b_2^{-1},\] and so $b_1^{\prime -1} b_1 \in B_2^2$. Since different values of $b'_1$ give different values of $b_1^{\prime - 1}b_1$, it follows that $|B_1^2 \cap B_2^2|\geq |F|$.
Note, however, that
\[ B_1^2 = (A^2 \cap S_i)^2 \subseteq A^4 \cap S_i,\] whilst
\[ B_2^2 = a_i (A^2 \cap T)^2 a_i^{-1} \subseteq a_i T a_i^{-1},\]
whence
\[ B_1^2 \cap B_2^2 \subseteq A^4 \cap S_i \cap a_i T a_i^{-1} = A^4 \cap S_{i+1}.\]
Therefore $|A^4 \cap S_{i+1}| \geq |F|$. Since $A$ is a $K$-approximate group, $A^{4}$ is covered by $K^{3}$-translates of $A$. In particular there is some $x$ such that $|A \cap S_{i+1}x| \geq K^{-3}|F|$, and this immediately implies that $|A^2 \cap S_{i+1}| \geq K^{-3}|F| \geq \delta_{i+1}|A|$, the desire to have this last inequality hold being the reason for our particular choice of $\delta_{i+1}$. This ends the proof of the induction step and hence the proof of the corollary.
\endproof

\emph{Remarks.} Suppose that $A \subseteq \U_n(\C)$ is a symmetric set satisfying the \emph{small tripling} condition $|A^3| \leq K|A|$. Then the conclusion of Corollary \ref{main-corollary} still holds, since then $A^3$ is a $K^{C}$-approximate group containing $A$. This follows from standard multiplicative combinatorics (see, for example, Proposition 3.1 in \cite{bg1}). As a consequence we obtain the following corollary.



\begin{corollary} \label{corollary2} Suppose that $A \subseteq U_n(\C)$ is a symmetric subset with $|A^3| \leq K|A|$ and that the closure of the subgroup $\langle A \rangle$ is a connected subgroup of $\U_n(\C)$ with no connected abelian normal subgroup. Then $|A| \leq n^{Cn^4} K^{Cn^2}$.\end{corollary}

\begin{proof} The set $A^{3}$ is a $K^C$-approximate group, and so by Corollary \ref{main-corollary}  it must be contained in $N(S)$, where $S$ is a connected abelian subgroup of $\U_n(\C)$. By our assumption, $G:=\overline{\langle A \rangle}$ is a connected semisimple compact group with dimension $\leq n^2$. It is well-known (e.g. see \cite{tomdieck}) that the centre of a connected semisimple compact Lie group of dimension $d$ is finite and in fact of size at most $d$. Since $S \cap G$ is a finite normal subgroup of $G$, it is central (this follows by connectedness of $G$, since the map $G \rightarrow G, g \mapsto gxg^{-1}$ is continuous and takes only finitely many values if $x$ belongs to a finite normal subgroup) and thus of size at most $n^2$. By Corollary \ref{main-corollary}, $|A^2 \cap S| \geq n^{-Cn^4}K^{-Cn^2}|A|$, and so the result follows immediately.\end{proof}

If $A$ is only assumed to have small doubling, i.e. $|A^2| \leq K|A|$, then it follows from the non-commutative Balog-Szemer\'edi-Gowers lemma (see \cite{tao-noncommutative}) that $A$ is $K^C$-controlled by a $K^C$-approximate subgroup. In particular, applying Theorem \ref{mainthm}, we conclude that $A$ is contained in $n^{Cn^3}K^{Cn}$ cosets of a connected abelian subgroup of $\U_n(\C)$.

\section{On Gromov's theorem}

\label{gromov-sec}

In this section we show how our main result gives a new elementary proof of the fact that non-virtually abelian subgroups of $\U_n(\C)$ cannot have polynomial growth, and in fact have growth at least $\exp(r^{\alpha})$.

Recall that a group $G$ has \emph{polynomial growth} with exponent $d$ if there is a finite symmetric set $\Sigma$ of generators such that one has the bound
\begin{equation}\label{poly-growth} |\Sigma^r| \leq Br^d\end{equation} for all $r \geq 1$, where $B = B_{\Sigma}$ does not depend on $r$. If one set $\Sigma$ of generators has this property then it is easy to see that any other set $\Sigma'$ does too, although $B_{\Sigma'}$ may be different. Thus polynomial growth is a well-defined property of the \emph{group}.

\begin{proposition}\label{prop5}
Suppose that $G \subseteq \U_n(\C)$ is a finitely generated group with polynomial growth. Then $G$ is virtually abelian.
\end{proposition}

\begin{proof}
Let $S$ be a generating set. There are clearly arbitrarily large $r$ for which
\begin{equation}\label{good-r-bound} |\Sigma^{7r}| \leq 8^d |\Sigma^r|,\end{equation} since if not the polynomial growth hypothesis would be violated. Call these values good, and suppose in what follows that $r$ is good.
By standard multiplicative combinatorics (see in particular Proposition 3.1 (v) in Part I of this series) it follows that $A := \Sigma^{3r}$ is a $K$-approximate group for some $K = O(1)^d$. By Theorem \ref{mainthm}, there is some abelian group $H \leq \U_n(\C)$ and a coset $Hx$ such that $|A \cap Hx| \geq c_{n,d}|A|$, where $c_{n,d} > 0$ depends only on $n$ and $d$. We therefore have
\begin{equation}\label{s6-bd} |\Sigma^{6r} \cap H| = |A^2 \cap H| \geq c_{n,d}|A| \geq c_{n,d} |\Sigma^r|.\end{equation}
Replacing $H$ by the subgroup generated by $A^2 \cap H$ (if necessary) we may assume without loss of generality that $H \leq G$. Assume that $[G : H] = \infty$. Then, since $\Sigma$ generates $G$, it is easy to see that $\Sigma^k$ meets at least $k$ different right cosets of $H$, for every integer $k \geq 1$. It follows from this observation and \eqref{s6-bd} that
\[ |\Sigma^{6r + k}| \geq k c_{n,d} |\Sigma^r|.\]
Choosing $k > 8^d/c_{n,d}$ and some good value of $r$ with $r > k$, we obtain a contradiction to \eqref{good-r-bound}. Thus we were wrong to assume that $[G : H] = \infty$, and this concludes the proof.
\end{proof}

One could run the above argument more carefully to get an explicit upper bound on $[G : H]$. However this observation is redundant here since it is known by rather easier arguments that any virtually abelian group $G \leq \U_n(\C)$ has an abelian subgroup $H$ with $[G : H] \leq F(n)$, where $F(n)=O(n!(n+1)!)$.  We offer a brief sketch proof of this fact in Appendix \ref{virtually-abelian-app}.

\vspace{.3cm}

Using Corollary \ref{corollary2}, one can also prove the following quantitative form of the above proposition.

\begin{proposition}\label{prop6}
Let $\Sigma$ be a finite symmetric subset of $\U_n(\C)$ and that $\langle \Sigma \rangle$ is not virtually abelian. Then $|\Sigma^r| \geq 2^{c r^{\alpha}}$ for all $r \geq 1$, where $\alpha>0$ depends only on $n$ and $c=c_\Sigma >0$.
\end{proposition}

\begin{proof} Let $G$ be the closure of the subgroup $\langle \Sigma \rangle$ generated by $\Sigma$, let $G^0$ its connected component of the identity, and $i:=[G:G^0]$.  There is no loss of generality in passing to subgroup $\langle \Sigma \rangle \cap G^0$. Indeed $\Sigma^{2i-1}$ contains a generating set for $\langle \Sigma \rangle \cap G^0$ (see e.g. \cite[Lemma C.1]{bgt-paper}) and we may replace $\Sigma$ by this subset. As a result, we may assume that $G$ is connected. Let $Z$ be its centre and write $\pi:G \rightarrow G/Z$ for the quotient. Then, for every $r\geq 1$, $\pi(\Sigma^r)$ generates a dense subgroup of the non-trivial connected centre-free compact Lie group $G/Z$. The contrapositive of Corollary \ref{corollary2} therefore applies and we obtain an $\eps = \eps(n) > 0$ for which $|\pi(\Sigma^{3k})| \geq |\pi(\Sigma^k)|^{1+\eps}$ for every $k \geq 1$. Iterating this clearly leads to a bound of the form $|\pi(\Sigma^r)| \geq 2^{cr^{\alpha}}$, which certainly implies the proposition.\end{proof}

Of course, much stronger results in this context are known. In fact from the Tits alternative \cite{tits}, the theorem of Milnor \cite{milnor} and Wolf \cite{wolf}, and the fact that every nilpotent subgroup of $\U_n(\C)$ is virtually abelian it follows that a finitely-generated subgroup $G \leq \U_n(\C)$ which is not virtually abelian has \emph{exponential} growth. Moreover, in view of the uniform Tits alternative \cite{breuillard-tits}, the exponential growth rate is even independent of $\Sigma$.

We remark that non polynomial growth for certain subgroups of $\GL_n(\C)$ was used as a key ingredient by Gromov in his original work \cite{gromov} and also, subsequently, by Kleiner \cite{kleiner}, who needed this fact only for subgroups of $\U_n(\C)$. Our arguments here may be inserted into Kleiner's work, thereby avoiding any appeal to the Tits alternative. It should be noted that Shalom and Tao \cite{shalom-tao} also avoid the Tits alternative in the relevant step of their variant of Kleiner's proof, appealing instead to the Solovay-Kitaev argument.

\appendix

\section{Simple facts from metric geometry}\label{bes-app}

We need some facts concerning covering by balls in certain metric spaces. If $(X,d)$ is a metric space then we write $B(x,r) = \{y \in X : d(x,y) < r\}$ for the \emph{open} ball of radius $r$ centred on $x$ and $\overline{B}(x,r) := \{ y \in X : d(x,y) \leq r\}$ for the corresponding closed ball.

\begin{definition}\label{besi-def}
Let $(X,d)$ be a metric space. We say that $X$ has the weak Besicovitch property with constant $k$ if the following is true. If $x_1,\dots,x_k \in X$ and if $r_1,\dots,r_k \in \R_{\geq 0}$ are such that the closed balls $\overline{B}(x_i,r_i)$ have nonempty intersection then there are distinct indices $i$ and $j$ such that $x_i$ lies in the open ball $B(x_j,r_j)$
\end{definition}

We call this the \emph{weak} Besicovitch property since it follows easily from the usual Besicovitch covering property as detailed, for example, in Theorem 1.1 of \cite{deG}. It seems to be somewhat weaker and easier to prove than that property, however. We shall write $k_{\bes}(X)$ for the smallest constant $k$ which works in the above definition.\vspace{11pt}

\emph{Example.} We have $k_{\bes}(\R^2) = 8$, where $\R^2$ is endowed with the Euclidean metric and identified with the complex plane. To see that $k_{\bes}(\R^2) \leq 8$, suppose that $x_1,\dots,x_8 \in \R^2$ and that $r_1,\dots,r_8 \in \R_{\geq 0}$. Let $z$ lie in the intersection of all eight of the closed balls $\overline{B}(x_i,r_i)$. Perhaps one of the $x_i$ coincides with $z$; if so, suppose it is $x_8$. By the pigeonhole principle there is some choice of $i,j$, $1 \leq i < j \leq 7$, such that the angle $\angle x_i z x_j$ is less than $\pi/3$; this means that $|x_i - x_j|$ is less than either $r_i \geq |x_i - z|$ or $r_j \geq |x_j - z|$, and hence that either $x_i \in B(x_j,r_j)$ or $x_j \in B(x_i, r_i)$. On the other hand it is clear by considering $x_j = e^{2\pi i j/6}$, $j = 1,2,\dots,6$, $x_7 = 0$ and $r_j = 1$ that $k_{\bes}(\R^2)$ is not \emph{less} than 8.\vspace{11pt}

It is not particularly difficult to adapt the preceding argument to establish the following.

\begin{lemma}\label{balls-besi} Suppose that $\R^n$ is endowed with the Euclidean metric. Then $k_{\bes}(\R^n) \leq 3^n+1$.
\end{lemma}
\proof By the argument just outlined for $\R^2=\C$, it suffices to show that if $3^n$ distinct points $y_1,\dots,y_m$ are taken on the unit sphere in $\R^n$ then there are distinct indices $i,j$ such that the angle $\angle y_i 0 y_j$ is less than $\pi/3$. But if there is no such pair of indices then the open spherical caps centred on $y_i$ and with radius $\pi/6$ are disjoint. We conclude by a volume-packing argument, considering the balls of radius $\frac{1}{2}$ centred on the points $y_i$ together with the one centred at the origin. There are at least $3^n+1$ of these balls, which are disjoint, have radius $\frac{1}{2}$, and are all contained in the ball of radius $\frac{3}{2}$ about the origin. This is impossible since $(3^n+1)/2^n> (3/2)^n$.\endproof

This has the following simple corollary, which we used in the proof of our main theorem.

\begin{corollary}
Let $(X,d)$ be the metric space consisting of the matrices $X = \Mat_n(\C)$ together with the distance induced from the Hilbert-Schmidt norm. Then the weak Besicovitch constant $k_{\bes}(X)$ is bounded by $3^{2n^2}+1$.\endproof
\end{corollary}

\section{Virtually abelian subgroups of $\U_n(\C)$.}
\label{virtually-abelian-app}

Our aim in this appendix is to outline a proof of the following statement.
\begin{proposition}\label{virtual-abelian}
Suppose that $G \leq \U_n(\C)$ be a virtually abelian group. Then there is a normal abelian subgroup $H \leq G$ with $[G: H] \leq O(n!(n+1)!)$.
\end{proposition}
\emph{Remark.} The rather strong bound we obtain relies heavily on Collins' s bound for Jordan's theorem \cite{collins}, which in turn depends on the Classification of Finite Simple Groups. Inputting softer proofs of Jordan's theorem (such as the one we gave in \S \ref{jordan-sec} of this paper) would give a vastly more elementary argument, but would lead to correspondingly cruder bounds of the form $\exp(Cn^C)$.\vspace{8pt}

\begin{proof}
Passing to the Zariski closure, we may assume without loss of generality that $G$ is an algebraic subgroup. Its connected component of the identity is a torus $S \subset \U_n(\C)$. The centraliser $Z(S)$ of this torus is a direct product of unitary groups $\U_m(\C)$ which are permuted by the normaliser $N(S)$. In particular, $[G:Z(S)]\leq [N(S):Z(S)] \leq n!$. According to a lemma of Platonov, for any algebraic subgroup $H \leq \GL_n(\C)$ there exists a finite subgroup $F$ such that $H=FH^0$, where $H^0$ is the Zariski connected component of the identity (see \cite[10.10]{wehrfritz}). Applying this to $H = G \cap Z(S)$, we get a finite subgroup $F\subset Z(S)$ such that $G \cap Z(S)=FS$. By Jordan's theorem and Collins's bound \cite{collins}, there is a normal abelian subgroup $F_0 \subseteq F$ of index $O((n+1)!)$. Now $F_0S$ is abelian and normal in $G$ and of index $O(n!(n+1)!)$.
\end{proof}

We conclude by remarking that simple examples show that no analogue of Proposition \ref{virtual-abelian} holds in $\GL_n(\C)$. Indeed the group $G \leq \GL_2(\C)$ consisting of all upper triangular matrices whose diagonal entries are $1$ and an $m$th root of unity is virtually abelian yet has no abelian subgroup of index less than $m$.

\end{document}